\documentclass[twoside,a4paper,reqno,11pt]{amsart} 
\usepackage{amsfonts, amsbsy, amsmath, amssymb, latexsym,hyperref, mathtools, bold-extra, url}
\usepackage{mathrsfs,array}
\usepackage[top=24mm,right=28mm,bottom=24mm,left=28mm]{geometry}
\usepackage{stmaryrd}
\usepackage{bm}
\usepackage[pdftex]{color,graphicx}

\renewcommand{\a}{\alpha}
\renewcommand{\b}{\beta}
\newcommand{\normeq}{\trianglelefteqslant}

 \renewcommand{\O}{\Omega}

 \renewcommand{\to}{\rightarrow}

\newcommand{\la}{\langle}
\newcommand{\ra}{\rangle}

\newcommand{\leqs}{\leqslant}
\newcommand{\geqs}{\geqslant}

 \newcommand{\vs}{\vspace{3mm}}

\DeclareMathOperator{\GL}{GL}

\DeclareMathOperator{\Fix}{Fix}

\newcommand{\F}{\mathbb F}

\makeatletter
\newcommand{\imod}[1]{\allowbreak\mkern4mu({\operator@font mod}\,\,#1)}
\makeatother

\newtheorem*{theorem*}{Theorem}
\newtheorem{athm}{Theorem}
 
\newtheorem*{conj*}{Conjecture}

\newtheorem{thm}{Theorem}[section] 
\newtheorem{prop}[thm]{Proposition} 
\newtheorem{lem}[thm]{Lemma}

\theoremstyle{definition}
\newtheorem{rem}[thm]{Remark}

\begin{document}

\author{Timothy C. Burness}
\address{T.C. Burness, School of Mathematics, University of Bristol, Bristol BS8 1UG, UK}
\email{t.burness@bristol.ac.uk}

\author{Hangyang Meng}
\address{H. Meng, Department of Mathematics and  Newtouch Center for Mathematics of Shanghai University, College of Science, Shanghai University, Shanghai 200444, P.R. China}
\email{hymeng2009@shu.edu.cn}

\title[On second maximal subgroups of finite groups]{On second maximal subgroups of finite groups}

\begin{abstract}
A second maximal subgroup of a finite group $G$ is a subgroup $H$ that is maximal in every maximal overgroup of $H$ in $G$. We prove that every maximal subgroup of a prime-index subgroup of a finite group is a second maximal subgroup.
\end{abstract}

\date{\today}

\maketitle

\section{Introduction}\label{s:intro}

Let $G$ be a finite group and let $H$ be a proper subgroup of $G$. We say that $H$ is a \emph{weak second maximal} subgroup if it is a maximal subgroup of a maximal subgroup of $G$. Properties of weak second maximal subgroups and their influence on the structure of the ambient group have been studied since the 1950s. For example, a theorem of Huppert \cite{Huppert1954} states that $G$ is supersolvable if every weak second maximal subgroup is normal. And work of Janko \cite{Janko1962} shows that $A_5$ and ${\rm SL}_2(5)$ are the only nonsolvable groups with the property that every weak second maximal subgroup is nilpotent. 

The term weak second maximal subgroup was first coined by Flavell in  \cite{Flavell}. In addition, he refers to a proper subgroup $H$ of $G$ as a \emph{second maximal} subgroup if it is maximal in \emph{every} maximal overgroup of $H$ in $G$. For example, if we take the subgroup $H = \la (1,2) \ra$ of $G = S_4$, then $H$ is a weak second maximal subgroup because it is maximal in $\la (1,2),(1,3)\ra$, but it is not a second maximal subgroup since it is not maximal in $\la (1,2),(1,3)(2,4)\ra$. We note that some authors use the terms \emph{$2$-maximal} and \emph{strictly $2$-maximal} in place of weak second maximal and second maximal, respectively (see \cite{KMS}, for example).

Given a proper subgroup $H$ of $G$, it is natural to consider the number of maximal subgroups of $G$ containing $H$, which we denote by $m(G,H)$, and there is an extensive literature in this direction. In \cite{Flavell}, for example, Flavell reveals a relationship between $m(G,H)$ and the indices of maximal overgroups of a second maximal subgroup $H$ of $G$. Indeed, \cite[Theorem A]{Flavell} establishes the upper bound
\[
m(G,H) \leqs 1 + \max\{ |G:M| \,:\, \mbox{$M$ is a maximal subgroup of $G$ containing $H$}\}
\]
for every second maximal subgroup $H$ of a finite group $G$, together with a detailed description of the pairs $(G,H)$ for which equality holds ($H/H_G$ is solvable in each case, where $H_G$ denotes the core of $H$ in $G$). In the same paper, Flavell asked whether or not the same upper bound holds when $H$ is a \emph{weak} second maximal subgroup. For solvable groups, this was answered in the affirmative by Meng and Guo \cite{MG-2}.

In this paper, we seek to shed new light on the following natural problem: when does a weak second maximal subgroup $H$ of $G$ have the stronger second maximal subgroup property? In particular, we would like to understand how this is related to the index of a maximal overgroup of $H$.
A recent result in this direction due to Konovalova, Monakhov and Sokhor \cite{KMS} states that if $H$ is a weak second maximal subgroup of $G$ and $M$ is a maximal overgroup of $H$, then $H$ is a second maximal subgroup if $M$ is normal in $G$, or if $G$ is $p$-solvable and $|G:M|=p$ for some prime $p$. Our main theorem shows that the $p$-solvable hypothesis is not needed.

\begin{athm}\label{thm-max}
Let $G$ be a finite group and let $H$ be a maximal subgroup of a prime-index subgroup of $G$. Then $H$ is a second maximal subgroup of $G$.
\end{athm}

In order to discuss possible extensions of Theorem \ref{thm-max}, it will be convenient to introduce the following notation. For each positive integer $n \geqs 2$, let $(\mathcal{P}_n)$ be the following assertion:
\[
\begin{array}{l}
\mbox{\emph{For any finite group $G$ with a maximal subgroup $M$ of index $n$, every }} \\
\mbox{\emph{maximal subgroup of $M$ is a second maximal subgroup of $G$.}}
\end{array}
\]
So Theorem \ref{thm-max} states that $(\mathcal{P}_n)$ holds for every prime $n$. And the previous example with $G = S_4$, $M = \la (1,2),(1,3) \ra$ and $H = \la (1,2) \ra$ shows that $(\mathcal{P}_4)$ is false.

In fact, it is straightforward to show that $(\mathcal{P}_n)$ is false for every composite even integer and every composite prime power:

\begin{itemize}\addtolength{\itemsep}{0.2\baselineskip}
\item[{\rm (a)}] Suppose $n = 2m \geqs 4$ is even and let $G = S_n$, $M = S_{n-1}$ and $K = S_m \wr S_2$, so $M$ and $K$ are maximal in $G$ and $|G:M| = n$. Then we can take a subgroup $H = S_m \times S_{m-1}< S_m \times S_m < K$ that is maximal in $M$, so $H$ is a weak second maximal subgroup, but it is not a second maximal subgroup.

\item[{\rm (b)}] Let $n = p^d$ be a prime power with $d \geqs 2$. Let $V = \mathbb{F}_p^d$ and set $G = {\rm AGL}(V) = V{:}{\rm GL}(V)$, $M = {\rm GL}(V)$ and $K = V{:}H$, where $H$ is the stabilizer in $M$ of a $1$-dimensional subspace of $V$. Then the irreducibility of $M$ on $V$ implies that $M$ is maximal in $G$ with $|G:M| = |V| = n$, and $K$ is maximal in $G$ since $H$ is maximal in $M$. However, $H$ acts reducibly on $V$ since $d \geqs 2$ and thus $H$ is not maximal in $K$. 
\end{itemize}

The problem for odd composite integers $n \geqs 15$ that are not prime powers is more subtle. In the statement of the following result, we set 
\[
\mathcal{A} = \{ n \in \mathbb{N} \,:\, \mbox{$S_n$ and $A_n$ are the only primitive subgroups of $S_n$} \}.
\]
With the aid of {\sc Magma} \cite{magma}, it is easy to check that  
\[
\{n \in \mathcal{A} \,:\, n \leqs 100\} = \{ 34, 39, 46, 51, 58, 69, 70, 75, 76, 86, 87, 88, 92, 93, 94, 95, 96, 99 \}.
\]

\begin{athm}\label{thm-max2}
Let $G$ be a finite group and let $H$ be a maximal subgroup of a maximal subgroup $M$ of $G$ such that $|G:M| = n \in \mathcal{A}$ is odd. Then $H$ is a second maximal subgroup of $G$.
\end{athm}

So this shows that $(\mathcal{P}_n)$ is true for every odd integer $n \in \mathcal{A}$. Now if we define  
\[
\delta(\mathcal{A}) = \lim_{n \to \infty} \frac{|\mathcal{A} \cap \{1, \ldots,n\}|}{n}
\]
to be the natural density of $\mathcal{A}$ as a subset of $\mathbb{N}$, then a well known theorem of Cameron, Neumann and Teague \cite{CNT} states that $\delta(\mathcal{A}) = 1$. This means that $S_n$ and $A_n$ are the only primitive subgroups of $S_n$ for almost every positive integer $n$. And the same conclusion holds if we only consider the set of odd positive integers. So Theorem \ref{thm-max2} implies that $(\mathcal{P}_n)$ is true for almost every odd positive integer. We refer the reader to Remark \ref{r:odd} for further discussion in this direction.

Our proof of Theorem \ref{thm-max} uses a result of Meng and Guo \cite[Lemma 1]{MG} to reduce the problem to a question concerning transitive permutation groups of prime degree. The result we need in this setting is stated as Theorem \ref{thm:main} in Section \ref{s:A} and we will use it to establish Theorem \ref{thm-max} by eliminating the existence of a minimal counterexample. So the main bulk of the argument concerns the proof of Theorem \ref{thm:main}. By a classical result of Burnside, every transitive permutation of prime degree is either $2$-transitive or solvable (and hence affine), and it is easy to read off the complete list of cases by appealing to the classification of the finite $2$-transitive groups. Note that by invoking the latter result, our argument depends on the Classification of Finite Simple Groups. The case where $G$ is an almost simple group with socle ${\rm PSL}_d(q)$ and degree $(q^d-1)/(q-1)$ requires the most effort.

We will also use \cite[Lemma 1]{MG} in our proof of Theorem \ref{thm-max2} to show that if $(\mathcal{P}_n)$ fails for a given positive integer $n$, then this must be witnessed by a primitive subgroup $G \leqs S_n$. And then it just remains to show that there are no such examples when $n$ is odd and $G = S_n$ or $A_n$, which is how we proceed.

Finally, we anticipate that Theorems \ref{thm-max} and \ref{thm-max2} will be useful for studying the so-called \emph{WSM-groups}. These are the finite groups for which every weak second maximal subgroup is second maximal. The solvable WSM-groups have been studied by Meng and Guo in \cite{MG}, but a complete classification remains out of reach. This is also related to Problem 19.54 in the Kourovka Notebook \cite{Khukhro2026}, which asks for a description of the chief factors of a WSM-group.

\vs

\noindent \textbf{Acknowledgements.} Burness thanks the Department of Mathematics at Shanghai University for their generous hospitality during a research visit in June 2026. This work is supported by National Natural Science Foundation of China (12471018) and Natural Science Foundation of Shanghai (24ZR1422800). 

\section{Proof of Theorem \ref{thm-max}}\label{s:A}

In this section, we will prove Theorem \ref{thm-max}. As we will see, the key ingredient is the following result concerning transitive permutation groups of prime degree.

\begin{thm}\label{thm:main}
Let $G \leqs {\rm Sym}(\Omega)$ be a transitive permutation group of prime degree with point stabilizer $G_{\a}$ and let $M$ be a maximal subgroup of $G$. If $M_{\a}$ is maximal in $G_{\a}$, then it is also  maximal in $M$.  
\end{thm}

As in Theorem \ref{thm:main}, let $G \leqs {\rm Sym}(\O)$ be a transitive permutation group with point stabilizer $H = G_{\a}$ and prime degree $r$. Since every transitive group of prime degree is primitive, it follows that $H$ is a core-free maximal subgroup of $G$. By a classical theorem of Burnside, $G$ is either solvable or $2$-transitive, so by inspecting the list of $2$-transitive permutation groups (see \cite[Chapter 7]{Cam} for example), which relies on the Classification of Finite Simple Groups, we deduce that one of the following holds (also see \cite[Theorem 3]{Jones}):

\begin{itemize}\addtolength{\itemsep}{0.2\baselineskip}
\item[{\rm (a)}] $C_r \normeq G \leqs {\rm AGL}_1(r)$.
\item[{\rm (b)}] $G = A_r$ or $S_r$ with $r \geqs 5$.
\item[{\rm (c)}] $(G,r) = ({\rm PSL}_{2}(11),11)$, $({\rm M}_{11},11)$ or $({\rm M}_{23},23)$.
\item[{\rm (d)}] ${\rm PSL}_d(q) \normeq G \leqs {\rm P\Gamma L}_d(q)$ with $d \geqs 2$ and $r = (q^d-1)/(q-1)$.
\end{itemize}

Let $M$ be a maximal subgroup of $G$ and assume that $M_{\a}$ is maximal in $H = G_{\a}$. Our aim is to show that $M_{\a}$ is also maximal in $M$. With this goal in mind, let us record two basic observations:
\begin{itemize}\addtolength{\itemsep}{0.2\baselineskip}
\item[{\rm (i)}] Clearly, if $M$ is transitive on $\O$ then $|M:M_{\a}| = r$ and thus $M_{\a}$ is maximal in $M$. Therefore, we may assume
that $M$ has two or more orbits on $\O$.
\item[{\rm (ii)}] Let $\a^M$ be the $M$-orbit of $\a$. Since we are assuming $M_{\a}$ is maximal in $G_{\a}$, it follows that $|\a^M| \geqs 2$ and our goal is to show that $M$ acts primitively on $\a^M$.
\end{itemize}

We will now consider cases (a)-(d) in turn.

\begin{prop}\label{p:a}
The conclusion to Theorem \ref{thm:main} holds in case (a).
\end{prop}

\begin{proof}
Here $G = N{:}H \leqs {\rm AGL}_1(r)$ is a primitive affine group with socle $N = C_r$ and point stabilizer $H  = G_{\a} \leqs C_{r-1}$. In particular, $G$ is a Frobenius group with kernel $N$ and complement $H$. As explained above in (i), we may assume $M$ is intransitive on $\O$, which in turns means that $|M|$ is indivisible by $r$ (otherwise $M$ contains the unique Sylow $r$-subgroup $N$, which acts transitively on $\O$). Therefore, Hall's theorem implies that $M \leqs H^x < G$ for some $x \in N$, so the maximality of $M$ in $G$ forces $M=H^x$.

If $H^x=H$, then $M=H$ and thus $M_{\a}=G_{\a}$, a contradiction.  
Therefore, since $G$ is a Frobenius group, we have $M_{\a} = H \cap H^x = 1$ and so the maximality of $M_{\a}$ in $H$ forces $|H|$ to be a prime. In particular, $M$ has prime order and we conclude that $M_{\a} = 1$ is maximal in $M$, as required.
\end{proof}

\begin{prop}\label{p:b}
The conclusion to Theorem \ref{thm:main} holds in case (b).
\end{prop}

\begin{proof}
Here $G = S_r$ or $A_r$ with $r \geqs 5$. Given a subset $\Delta$ of $\O$, let $S_{\Delta}$ be the symmetric group on $\Delta$. Since we may assume $M$ is intransitive on $\O$, it follows that 
\[
M=(S_{\Delta_1} \times S_{\Delta_2})\cap G,
\]
where the $\Delta_i$ are nonempty and $\Omega=\Delta_1 \cup \Delta_2$ is a partition. Without loss of generality, we may assume that $\alpha \in \Delta_1$, in which case $\Delta_1=\alpha^M$ has at least two elements. 

Clearly, if $G = S_r$, or if $G = A_r$ and $|\Delta_1| \geqs 3$, then $M$ acts primitively on $\Delta_1$ and we conclude that $M_{\a}$ is maximal in $M$. And similarly, if $G = A_r$ and $|\Delta_1|=2$, say $\Delta_1=\{\alpha, \alpha'\}$, then $|\Delta_{2}|=r-2\geqs 3$ and we can choose distinct points $\b,\b' \in \Delta_2$ so that $M$ contains the double transposition $(\alpha, \alpha')(\beta, \beta')$. So once again, $M$ acts primitively on $\Delta_1$ and the result follows.
\end{proof}

\begin{prop}\label{p:c}
The conclusion to Theorem \ref{thm:main} holds in case (c).
\end{prop}

\begin{proof}
It is straightforward to use {\sc Magma} \cite{magma} to handle these special cases.

To do this, we first construct $G$ as a permutation group on $\O = G/H$ and we construct a set of representatives of the conjugacy classes of maximal subgroups of $G$. For each representative $M$, we calculate the orbits $\O_1, \ldots, \O_t$ of $M$ on $\O$. Then for each $i$, we check that if $M$ acts imprimitively on $\O_i$ then $|\O_i|$ does not coincide with the index of a maximal subgroup of $H$. This allows us to conclude that $M$ acts primitively on $\a^M$ whenever $M_{\a}$ is a maximal subgroup of $G_{\a}$, as required.
\end{proof}

To complete the proof of Theorem \ref{thm:main}, it just remains to consider case (d) above. Here $G$ is an almost simple group with socle $T = {\rm PSL}_d(q)$, where $q = p^f$ for some prime $p$. Since we are assuming
\[
r=\frac{q^d-1}{q-1}
\]
is a prime, it follows that $d$ is a prime. Moreover, $\gcd(d,q-1)=1$.  Indeed, if $d$ divides $q-1$, then $d$ divides $r$ and thus $d = 1+q+\cdots+q^{d-1}$, which is absurd. Therefore, $T = {\rm SL}_d(q)$ and we have $G = T{:}\la \varphi \ra$, where $\varphi$ is a field automorphism of $T$ of order $e$ and $e \geqs 1$ is a divisor of $f$.

In order to handle this case, we need to introduce some additional notation, and we require a preliminary lemma.

Let $V = \mathbb{F}_q^d$ be the natural module for $T = {\rm SL}_d(q)$  and note that we may identify $\O$ with the set $\mathbb{P}(V)$ of $1$-dimensional subspaces of $V$. Fix a $1$-space $\a \in \O$ and observe that $T_{\a} = U{:}L$ is a maximal parabolic subgroup of $T$, with unipotent radical $U$ and Levi subgroup $L = {\rm GL}_{d-1}(q)$. Here $U$ is an elementary abelian $p$-group of order $q^{d-1}$, which we may identify with the natural module for $L$. In particular, we have $H = G_{\a} = U{:}(L{:}\la \varphi\ra)$ and $U$ is the unique minimal normal subgroup of $H$.

Visibly, $L$ preserves a decomposition $V = \a \oplus W$, where $W$ has dimension $d-1$. In terms of this decomposition, we can write
\begin{equation}\label{e:U}
U = \{ u_f \,:\, f \in {\rm Hom}_{\mathbb{F}_q}(W,\a) \},\;\; L = \{\ell_g \,:\, g \in {\rm GL}(W)\}
\end{equation}
as subgroups of ${\rm SL}(V)$, where
\[
u_f(a+w) = a+w+f(w),\;\; \ell_g(a+w) = (\det g)^{-1}a + g(w)
\]
for all $a \in \a$, $w \in W$.

Let $T_W$ be the setwise stabilizer in $T$ of the hyperplane $W$; this is a maximal parabolic subgroup of $T$ with Levi decomposition $T_W = Q{:}L$. For a concrete description, we can take
\begin{equation}\label{e:Q}
Q=\{ v_h \, : \, h \in {\rm Hom}_{\F_q}(\alpha,W)\},\;\; v_h(a+w)=a+w+h(a)
\end{equation}
for all $a\in \a$, $w \in W$. Note that $T_{\a}$ and $T_W$ are conjugate subgroups in ${\rm Aut}(T)$ with the property $T_{\a} \cap T_W = L$.

\begin{lem}\label{l:psl}
In terms of the above notation, the following statements hold:
\begin{itemize}\addtolength{\itemsep}{0.2\baselineskip}
\item[{\rm (i)}] If $J$ is a complement of $U$ in $T_{\a}$, then $J = L^x$ for some $x \in U$.
\item[{\rm (ii)}] If $d \geqs 3$ and $S$ is a subgroup of $T$ containing $L$, then $S = L$, $T_\alpha$, $T_W$ or $T$. 
\item[{\rm (iii)}] The unipotent radical $Q$ of $T_W$ acts regularly on the set $\O \setminus \mathbb{P}(W)$.
\end{itemize}
\end{lem}

\begin{proof}
First consider (i). We need to show that the first cohomology group $H^1(L,U)$ is trivial. For $d=2$, this is clear since $|U| = q$ and $|L|=q-1$ are coprime. Now assume $d \geqs 3$. If $q \geqs 3$ then $Z(L)$ contains a nontrivial element acting as a scalar on $U$, which forces $H^1(L,U) = 0$ (see \cite[Proposition 2.7(b)]{Sah1968}, for example, which is sometimes referred to as Sah's Lemma). Finally, if $q=2$ then $L = {\rm SL}_{d-1}(2)$ and the main theorem of \cite{Bell} gives $H^1(L,U) = 0$ (note that $H^1(L,U) \ne 0$ when $(d,q) = (4,2)$, but this case does not arise since $d$ is a prime in the setting we are interested in).

Next let us turn to (ii), so $d \geqs 3$. First observe that $T_{\alpha} = U{:}L$ and $T_{W} = Q{:}L$, where $L$ acts irreducibly on $U$ and $Q$. Therefore, $L$ is maximal in both $T_{\a}$ and $T_W$.

Let $S \ne L$ be an overgroup of $L$ in $T$. If $S$ acts reducibly on $V$, then $L<S\leqs T_X<T$ for some subspace $X$ of $V$. But $\a$ and $W$ are the only proper nonzero subspaces of $V$ preserved by $L$, so $X \in \{\a,W\}$ and then the maximality of $L$ in $T_X$ forces $S = T_X$. Now assume $S$ is irreducible on $V$. By appealing to Aschbacher's theorem on the subgroup structure of finite classical groups (see \cite{asch}), it is straightforward to show that $L$ is not contained in a proper irreducible subgroup of $T$, so $S = T$ and the result follows. Alternatively, if $S$ is irreducible then there exists an element $x \in S \setminus (T_{\a} \cup T_W)$ and we can show that $T = \la L, x \ra$, which means that $S = T$.  

Finally, let us consider (iii). Fix a $1$-space $\b = \la a+w \ra$ in $\O \setminus \mathbb{P}(W)$, where $0 \ne a \in \a$ and $w \in W$. Working with the explicit description of $Q$ in \eqref{e:Q}, let $1 \ne v_h \in Q$ and note that $0 \ne h \in {\rm Hom}_{\mathbb{F}_q}(\a,W)$. Then 
\[
v_h(a+w) = a + (w+h(a)) \not\in \b,
\]
so $Q_{\b} = 1$ and the result follows since $|Q| = |\O \setminus \mathbb{P}(W)| = q^{d-1}$.
\end{proof}

\begin{prop}\label{p:d}
The conclusion to Theorem \ref{thm:main} holds in case (d).
\end{prop}

\begin{proof}
Let $M$ be a maximal subgroup of $G$ such that $M_{\a}$ is maximal in $H = G_{\a}$. Set $J = M \cap T$ and note that $J \normeq M$. Recall that we may as well assume $M$ acts intransitively on $\O$, so the transitivity of $T$ (as a nontrivial normal subgroup of the primitive group $G$) implies that $M$ is core-free and thus $G = TM$. In particular, $J$ is a proper subgroup of $T$. There are now two cases to consider.

\vs

\noindent \emph{Case 1.} $U \leqs M_{\a}$

\vs

First assume $M_{\a}$ contains $U$, which is the unipotent radical of $T_{\a}$. Then
\[
U\leqs M_{\a} \cap T \leqs M \cap T=J
\] 
and thus $J$ contains the unipotent radical of a parabolic subgroup of $T$. So by the main theorem of \cite{Timm}, it follows that $J$ is contained in a maximal parabolic subgroup $T_X$ of $T$, where $X$ is a proper nonzero subspace of $V$.

For each $g \in M$ we have $J = J^g \leqs T_{X^g}$ since $J \normeq M$. It follows that $J \leqs T_{Y}$, where
\[
Y =\bigcap_{g \in M} X^g.
\]
We claim that $\a \subseteq Y$, so $Y$ is nonzero.

Seeking a contradiction, suppose $\alpha \nsubseteq X^g$ for some $g \in M$. Then by considering the decomposition $V = \a \oplus W$, there exists a vector $v = a+w \in X^g$ with $a \in \a$ and $0 \ne w \in W$. Fix $f \in {\rm Hom}_{\F_q}(W,\alpha)$ such that $f(w) \ne 0$, so we have $\a = \la f(w) \ra$. Since $X^g$ is $J$-invariant it is also $U$-invariant, so by working with the description of $U$ in \eqref{e:U} we see that
\[
f(w) = u_f(v) - v \in X^g
\]
and thus $\a \subseteq X^g$, a contradiction. 

Therefore, $Y$ is a proper nonzero $M$-invariant $k$-space containing $\a$, so $M \leqs G_{Y} <G$ and the maximality of $M$ implies that $M = G_Y$ and thus $\alpha^M \subseteq \mathbb{P}(Y)$. If $Y = \a$ then  $M=H$ and $M_{\a} = H$, which is a contradiction (since we are assuming $M_{\a}$ is maximal in $H = G_{\a}$). Therefore, $\a$ is a proper subspace of $Y$ and thus $2 \leqs k < d$.

Finally, let us observe that the parabolic subgroup $T_{Y}$ induces the full linear group $\GL(Y)$ on $Y$. Therefore, since $T_Y \leqs G_Y=M$, it follows that $\alpha^M=\mathbb{P}(Y)$. Moreover, the action of $M$ on $\a^M$ is $2$-transitive and therefore primitive. This completes the proof in Case 1.

\vs

\noindent \emph{Case 2.} $U \not\leqs M_{\a}$

\vs

To complete the proof, let us assume $U \not\leqs M_{\a}$. By the maximality of $M_{\a}$ in $G_{\a}$ we have $G_{\a} = UM_{\a}$.  Since $U \cap M_{\a}$ is normalized by $M_{\a}$ and centralized by the
abelian group $U$, we deduce that $U \cap M_{\a} \normeq G_{\a}$.  Therefore, $U \cap M_{\a} = 1$ since $U$ is a minimal normal subgroup of $G_{\a}$, whence $M_{\a}$ is a complement to $U$ in $G_{\a}$. In addition, 
\[
T_{\alpha}=G_{\a} \cap T=(UM_{\a})\cap T=U(M_{\a}\cap T)
\]
and thus Lemma \ref{l:psl}(i) implies that $M_{\a} \cap T$ and $L$  are
$U$-conjugate. Therefore, without loss of generality, we may assume that $M_{\a} \cap T=L$. We now divide the remainder of the argument  into two subcases according to $d$.

\vs

\noindent \emph{Case 2(a).} $d \geqs 3$

\vs

Suppose $d \geqs 3$. Since $J=M \cap T$ contains $L = M_{\a} \cap T$, Lemma \ref{l:psl}(ii) implies that $J = L$, $T_{\a}$, $T_W$ or $T$. Of course, if $J = T$ then $M$ is transitive on $\Omega$, contrary to our earlier assumption. Next suppose $J \leqs T_{\a}$, so $J$ fixes $\a$.
Since $J \normeq M$ we have $J=J^g$ for all $g \in M$, so $J$ fixes every point in $\a^M$. Since $d \geqs 3$, it is easy to see that 
$\Fix_{\Omega}(L)=\{\alpha\}$, so 
\[
\alpha^M \subseteq \Fix_{\Omega}(J) \subseteq \Fix_{\Omega}(L)=\{\alpha\}
\]
and thus $\alpha^M=\{\alpha\}$. This means that $M \leqs G_{\alpha}$ and hence $M_{\a} = G_{\a}$, which is a contradiction.

To complete the argument in Case 2(a), we may assume $J = T_W =  Q{:}L$ as above. Since $L$ acts transitively on $\mathbb{P}(W)$ and $Q$ is transitive on $\O \setminus \mathbb{P}(W)$, it follows that $J$ has exactly two orbits on $\O$, namely $\mathbb{P}(W)$ and $\O \setminus \mathbb{P}(W)$, which must coincide with the orbits of $M$ on $\O$ since we are assuming $M$ is intransitive. In particular, we have 
\[
|M:M_{\a}| = |\O \setminus \mathbb{P}(W)| = q^{d-1}.
\]
Since $M$ stabilizes $W$ we have $M \leqs G_W$ and thus $M=G_W$ by the maximality of $M$ in $G$. By Lemma \ref{l:psl}(iii), $Q$ acts regularly on $\Omega \setminus \mathbb{P}(W)$, so $Q_{\a} = 1$ and thus $Q \cap M_{\a} = 1$. It follows that $M = G_W = Q{:}M_{\a}$ with $Q \normeq M$. Finally, since $M_{\a}$ acts irreducibly on $Q$, we conclude that $M_{\a}$ is a maximal subgroup of $M$, as required.

\vs

\noindent \emph{Case 2(b).} $d=2$

\vs

Finally, let us assume $d=2$, in which case $r = q+1$ is a Fermat prime and $q = 2^{2^m}$ with $m \geqs 1$. If $q = 4$ then $T \cong A_5$ and we have already handled this possibility in Proposition \ref{p:b}. So we may assume $q \geqs 16$. Note that $T_{\a} = U{:}L$ is a Borel subgroup with $|U| = q$, $L = C_{q-1}$ and $\Fix_{\Omega}(L)=\{\alpha, W\}$. In addition, it is clear that $T_{\a}$ and $T_W$ are the only Borel subgroups of $T$ containing $L$, corresponding to the two points fixed by $L$.

In fact, by inspecting Dickson's list of the maximal subgroups of $T$ (see \cite[Tables 8.1, 8.2]{BHR}, for example), we observe that $L$ has exactly three maximal overgroups in $T$, namely $T_{\a}$, $T_W$ and $N_L(T) = L.2 = D_{2(q-1)}$. For example, $L$ is not contained in a subfield subgroup ${\rm SL}_2(q_0)$ with $q=q_0^2$ since the latter does not contain an element of order $q-1$. Since $J = M \cap T$ contains $L$, it follows that $J \leqs T_{\a}$, $T_W$ or $N_T(L)$.

First assume $J \not\leqs N_T(L)$, so $L < J \leqs T_\gamma$ for some $\gamma \in \{\alpha, W\}$ and the maximality of $L$ in $T_{\gamma}$ forces $J = T_{\gamma}$. Since $J \normeq M$ and $\Fix_{\Omega}(T_{\gamma})=\{\gamma\}$, 
we deduce that $M \leqs G_\gamma$ and so the maximality of $M$ implies that $M=G_\gamma$. If $\gamma=\alpha$, then $M_{\a} = G_{\a}$ and we reach a contradiction. Therefore, $\gamma=W$ and we have $M = Q{:}M_{\a}$, where $Q$ is the unipotent radical of $M = G_W$. And since $M_{\a}$ acts irreducibly on $Q$, we conclude that $M_{\a}$ is maximal in $M$.

Finally, let us assume $L \leqs J \leqs N_T(L)$, so $J = L$ or $N_T(L) = L.2$. In both cases, $L$ is characteristic in $J$ (recall that
$q-1$ is odd). Since $J \normeq M$ we deduce that $M \leqs N_G(L)$ and thus $M=N_G(L)$ by maximality. Then $M$ acts transitively on 
$\Fix_{\Omega}(L)=\{\alpha, W\}$ and $|M:M_{\a}| = 2$. In particular, $M_{\a}$ is maximal in $M$ and the proof of the proposition is complete.
\end{proof}

This completes the proof of Theorem \ref{thm:main} and we will now use it  to prove Theorem \ref{thm-max}. This relies on the following result, which is \cite[Lemma 1]{MG}. In the statement, we write $H_G$ and $M_G$ for the cores of $H$ and $M$, respectively (so $H_G$ is the largest normal subgroup of $G$ contained in $H$).

\begin{lem}\label{lem:core}
Let $G$ be a finite group and let $H$ be a proper subgroup of $G$. Suppose $M$ and $K$ are maximal overgroups of $H$ in $G$ such that $H$ is maximal in $M$, but not maximal in $K$. Then $H = M\cap K$ and  $H_G = M_G$.
\end{lem}

We are now ready to prove Theorem \ref{thm-max}.

\begin{proof}[Proof of Theorem \ref{thm-max}]
Suppose $G$ is a counterexample of minimal order, so $G$ has subgroups $H$, $M$ and $K$, where $M$ and $K$ are maximal overgroups of $H$ in $G$ with $|G:M|$ a prime, and $H$ is maximal in $M$, but not in $K$. Then $M \neq K$ and $H=M \cap K$. In addition, the minimality of $|G|$ implies that $H$ is core-free (otherwise we could pass to the quotient $G/H_G$), so $M$ is also core-free by Lemma \ref{lem:core}. 

Setting $\O = G/M$, we can now view $G \leqs {\rm Sym}(\O)$ as a transitive permutation group of prime degree with point stabilizer $G_{\a} = M$. Since $H=M \cap K=K_{\alpha}$ is maximal in $M$ by hypothesis, Theorem \ref{thm:main} now implies that $H$ is maximal in $K$. This final contradiction completes the proof of the theorem. 
\end{proof}

\section{Proof of Theorem \ref{thm-max2}}\label{s:B}

In this final section, we will prove Theorem \ref{thm-max2}. To do this, let $n$ be a positive integer and suppose the assertion $(\mathcal{P}_n)$ is false. So this means that there exists a finite group $G$ with a proper subgroup $H$ contained in maximal subgroups $M$ and $K$ of $G$ with the property that $|G:M|=n$ and $H$ is maximal in $M$ but not in $K$. By passing to $G/H_G$, we may assume that $H_G=1$ and then Lemma \ref{lem:core} implies that $M$ is core-free. So now we can view $G \leqs {\rm Sym}(\O)$ as a primitive permutation group on $\O = G/M$ with point stabilizer $M$. 

In other words, we have shown that if $(\mathcal{P}_n)$ is false, then this must be witnessed by a primitive subgroup $G \leqs S_n$. For $n$ odd, the next result rules out $S_n$ and $A_n$ as possible witnesses.

\begin{prop}\label{p:ansn}
Let $n \geqs 5$ be an odd integer and let $G = S_n$ or $A_n$ in its natural action of degree $n$. Fix a point stabilizer $M = G_{\a} = S_{n-1} \cap G$ and let $H$ be a maximal subgroup of $M$. Then $H$ is a second maximal subgroup of $G$.
\end{prop}

\begin{proof}
Let $K$ be a maximal subgroup of $G$ containing $H$. We need to prove that $H$ is maximal in $K$. We may as well assume that $K \ne M$, which implies that $H = M \cap K = K_{\a}$. Clearly, if $K$ acts primitively on $\O = \{1, \ldots, n\}$, then $K_{\a}$ is maximal in $K$. Therefore, we may assume that $K$ is either intransitive or imprimitive on $\O$.

First assume $K$ is intransitive and let $\Delta$ be the $K$-orbit containing $\alpha$. Set $\Gamma=\Omega\setminus\Delta$ and note that $|\Delta| \geqs 2$ and $K = G_{\Delta}$. Suppose $G = S_n$, so $K = S_\Delta \times S_\Gamma$ and $K_\alpha=(S_\Delta)_\alpha \times S_\Gamma$. Since $(S_\Delta)_\alpha$ is maximal in $S_\Delta$, we deduce that  $K_{\a}$ is maximal in $K$. Now assume $G = A_n$. If $|\Gamma|=1$ then $K_{\a} = A_{n-2}$ is a maximal subgroup of $K = A_{n-1}$. Now suppose $|\Gamma| \geqs 2$. Here the projection map $\pi : K \to S_{\Delta}$ is surjective and so the maximality of $(S_\Delta)_{\a}$ in $S_\Delta$ implies that the preimage $K_{\a} = \pi^{-1}((S_\Delta)_\alpha)$ is maximal in $K$.

Finally, suppose $K$ is transitive and imprimitive. Here $K$ is the stabilizer of a partition $\Pi$ of $\O$ into $b$ sets of size $a$, where $n = ab$ and $a,b \geqs 3$ are odd (since $n$ is odd). In particular, $K = (S_a \wr S_b) \cap G$. Let $\Gamma = \Gamma_1$ be the part in $\Pi$ containing $\a$ and let 
\[
J =(G_\alpha)_{\Gamma \setminus\{\alpha\}}
\]
be the setwise stabilizer of $\Gamma \setminus \{\alpha\}$ in $G_{\a}$. Since every element of $K_\alpha$ stabilizes $\Gamma$, it follows that 
$K_\alpha \leqs J<M$ and we note that $J$ is a proper subgroup of $M = G_{\a}$ since $\Gamma \setminus\{\alpha\}$ is a nonempty proper subset of $\Omega\setminus\{\alpha\}$. Since $a,b \geqs 3$, there exist distinct points $r,s,t,u$ such that $r,s \in \Gamma\setminus\{\a\}$ and $\a,u,v$ are contained in three distinct parts of $\Pi$. Then the double transposition 
$(r,s)(t,u)$ is in $J$, but it does not preserve $\Pi$, so it is not in $K_{\a}$ and we conclude that $H = K_{\alpha} < J <M$. But this is incompatible with the maximality of $H$ in $M$, so this case does not arise and the proof of the proposition is complete.
\end{proof}

\begin{proof}[Proof of Theorem \ref{thm-max2}]
Fix an odd integer $n \in \mathcal{A}$ and recall that if $(\mathcal{P}_n)$ is false, then this must be witnessed by a primitive subgroup of $S_n$. By definition of $\mathcal{A}$, the only primitive subgroups are $S_n$ and $A_n$, but neither group is a witness by Proposition \ref{p:ansn}. Therefore, we conclude that $(\mathcal{P}_n)$ is true.
\end{proof}

\begin{rem}\label{r:odd}
In view of Theorem \ref{thm-max2}, we know that $(\mathcal{P}_n)$ is true for every odd integer $n \in \mathcal{A}$. And by Theorem \ref{thm-max}, we also know that $(\mathcal{P}_n)$ holds for every odd prime $n$. However, it remains an open problem to completely determine the set of integers $n$ such that $(\mathcal{P}_n)$ holds. 

As explained above, to prove that $(\mathcal{P}_n)$ holds for a given integer $n$, it suffices to show that the conclusion holds for every primitive subgroup $G \leqs S_n$, where we take $M = G_{\a}$ to be the stabilizer of a point $\a \in \{1,\ldots, n\}$. Let us consider the first three composite odd integers $n \not\in \mathcal{A}$ that are not prime powers, so $n \in \{15,21,33\}$.

\begin{itemize}\addtolength{\itemsep}{0.2\baselineskip}
\item[{\rm (a)}] For $n = 15$ we can consider the primitive subgroup $G = S_6$, so $M = G_{\a} = S_2 \wr S_3$. With the aid of {\sc Magma}, we can show  that $M$ has a maximal subgroup $H = S_2 \times S_3$, which is contained in a maximal subgroup of $K = S_3 \wr S_2$. Since $K$ is maximal in $G$, we conclude that $(\mathcal{P}_{15})$ is false.

\item[{\rm (b)}] Next suppose $n = 21$. Here we take $G = S_7$ with $M = G_{\a} = S_5 \times S_2$, which allows us to identify $\O = \{1, \ldots, 21\}$ with the set of $2$-element subsets of $\{1, \ldots, 7\}$. If we choose a $2$-set $\b$ disjoint from $\a$, then $H = M_{\b} = (S_3 \times S_2) \times S_2$ is maximal in $M$ and we see that $H$ is also contained in a maximal subgroup $K = S_3 \times S_4$ of $G$. But we have $H < S_3 \times D_8 < K$, so $H$ is not maximal in $K$ and thus $(\mathcal{P}_{21})$ is false. 

\item[{\rm (c)}] Finally, suppose $n=33$. Here ${\rm PSL}_2(32)$ and ${\rm P\Gamma L}_2(32) = {\rm PSL}_2(32).5$ are the only primitive subgroups of $S_{33}$ that do not contain $A_{33}$. So we may write $G = {\rm PSL}_{2}(32).k$ with $k \in \{1,5\}$ and we note that $M = G_{\a} = 2^5{:}(31{:}k)$ is a Borel subgroup of $G$. Using {\sc Magma}, it is straightforward to check that every maximal subgroup of $M$ is a second maximal subgroup of $G$. It follows that $(\mathcal{P}_{33})$ holds, even though $33 \not\in \mathcal{A}$ is not a prime power.
\end{itemize}

We can also show that $(\mathcal{P}_{n})$ is false for $n \in \{35,45,55,65,85,91\}$ and true for $n \in \{57,63\}$. Let us also observe that it is easy to generalize the above construction in (b) to show that $(\mathcal{P}_n)$ is false for every integer $n = m(m-1)/2$ with $m \equiv 3 \imod{4}$ and $m \geqs 7$.
\end{rem}

\end{document}